\documentclass[graybox]{svmult}

\usepackage{type1cm}        
\usepackage{makeidx}         
\usepackage{graphicx}        
\usepackage{multicol}        
\usepackage[bottom]{footmisc}

\usepackage{dsfont} 
\usepackage{newtxtext}       %
\usepackage[varvw]{newtxmath}       

\usepackage{amsmath}
\usepackage{ulem}
\usepackage[cp1251]{inputenc} 

\usepackage[allcolors=blue,colorlinks=true, pdfstartview=FitH, linkcolor=blue]{hyperref}

\newcommand{\R}{\mathbb R}

\newcommand{\N}{\mathbb N}
\newcommand{\Z}{\mathbb Z}

\newcommand{\supp}{\text{supp}\,}

\makeindex             

\begin{document}

\title*{A note on continuity of strongly singular Calder\'on-Zygmund operators in Hardy-Morrey spaces}
\titlerunning{Continuity of strongly singular C-Z operators in Hardy-Morrey spaces}

\author{Marcelo de Almeida, Tiago Picon and Claudio Vasconcelos}
\institute{Marcelo de Almeida \at Departamento de Matem\'atica, Universidade Federal de Sergipe, Aracaj\'u, SE,
49000-000, Brasil \email{marcelo@mat.ufs.br} \and Tiago Picon \at Departamento de Computa\c{c}\~ao e Matem\'atica, Universidade de S\~ao Paulo, Ribeir\~ao Preto, SP, 14040-901, Brasil \email{picon@ffclrp.usp.br} \and Claudio Vasconcelos \at Departamento de Matem\'atica,  Universidade Federal de S\~ao Carlos, S\~ao Carlos, SP,  13565-905 , Brasil \email{claudio.vasconcelos@estudante.ufscar.br}}
%
%
\maketitle

\abstract{{In this note we address the continuity of strongly singular Calder\'on-Zygmund operators on Hardy-Morrey spaces $\mathcal{HM}_{q}^{\lambda}(\R^n)$, assuming weaker integral conditions on the associated kernel. Important examples that falls into this scope are pseudodifferential operators on the H\"ormander classes $OpS^{m}_{\sigma,\mu}(\R^{n})$ with $0<\sigma \leq 1$, $0 \leq \mu <1$, $\mu \leq \sigma$ and $m\leq -n(1-\sigma)/2$.}}

\section{Introduction}

J. \'Alvarez and M. Milman \cite{AlvarezMilman} introduced a new class of Calder\'on-Zygmund operators, called \textit{strongly singular Calder\'on-Zygmund operator}  {and established continuity of those operators in real Hardy space $H^q(\R^n)$.}   {More precisely,} a continuous function $K \in C(\R^{2n} \backslash \Delta)$, where $\Delta=\left\{(x,x):\, x \in \R^n \right\}$ is a $\delta$-kernel of type $\sigma$, if there exists  {{some}} $0<\delta\leq1$ and $0<\sigma \leq 1$ such that
\begin{equation} \label{pontual_kernel_tipo_sigma}
|K(x,y)-K(x,z)|+|K(y,x)-K(z,x)| \leq C \dfrac{|y-z|^{\delta}}{|x-z|^{n+ \frac{\delta}{\sigma}}},
\end{equation}
for all $|x-z| \geq 2|y-z|^{\sigma}$.  {A bounded  linear  operator} $T: \mathcal{S}(\R^{n}) \rightarrow \mathcal{S}'(\R^{n})$ is called a \textit{strongly singular Calder\'on-Zygmund operator},  %
if it is associated to a $\delta$-kernel of type $\sigma$ in the sense 
$
\langle Tf,g \rangle = \int{\int{K(x,y)f(y)g(x)dy}dx}$, for all  $ f,g \in \mathcal{S}(\R^n)$ with disjoint supports; it has bounded extension from $L^2(\R^n)$ to itself and in addition $T$ and $T^{\ast}$ extend to a continuous operator from $L^p(\R^n)$ to $L^2(\R^n)$, where 
$	\dfrac{1}{p} = \dfrac{1}{2} + \dfrac{\beta}{n}$ for some  $(1-\sigma) \dfrac{n}{2} \leq \beta < \dfrac{n}{2}$. 
When $\sigma=1$ and $\beta=0$ we recover the standard non-convolution Calder\'on-Zygmund operators (see \cite{Coifman-Meyer}). 

The authors in \cite[Theorem 2.2]{AlvarezMilman} established the continuity of those classes of operators in real Hardy spaces $H^q(\R^n)$ as follows: under the condition $T^{\ast}(1)=0$, strongly singular Calder\'on-Zygmund operators associated to a kernel satisfying \eqref{pontual_kernel_tipo_sigma} are bounded from $H^{q}(\R^n)$ to itself for every $q_{0}<q\leq 1$ where
	\begin{equation} \label{critical-index}
		\frac{1}{q_{0}} := \frac{1}{2}+\frac{\beta \left( \frac{\delta}{\sigma}+\frac{n}{2}\right)}{n\left(\frac{\delta}{\sigma}-\delta+\beta\right)} \ .
	\end{equation}
The case $q=q_{0}$ is still open, however the conclusion continues to hold replacing the target space by $L^{q_0}(\R^n)$ (see \cite[Theorem 3.9]{AlvarezMilmanVectorValued}).

 {In this note, we establish results on continuity of strongly singular Calder\'on-Zygmund operators on Hardy-Morrey spaces $\mathcal{HM}_{q}^{\lambda}(\R^n)$ assuming weaker integral conditions on the kernel, introduced by the second and third authors in \cite{VasconcelosPicon}. Let $0<  \sigma \leq 1$,  $r \geq 1$  and $\delta>0$. We say that $K(x,y)$ associated to $T$ is a $D_{\delta,r}$ kernel of type $\sigma$ if} 
\begin{equation} \label{Ls-hormander}
	\displaystyle{\left( \int_{C_j(z,\ell)}{|K(x,y)-K(x,z)|^r+|K(y,x)-K(z,x)|^rdx} \right)^{\frac{1}{r}} \lesssim |C_j(z,\ell)|^{\frac{1}{r}-1} \, 2^{-j\delta}}
\end{equation}
for $\ell \geq 1$ and
\begin{align} \label{1-hormander-2a}
	&\left( \int_{C_j(z,\ell^{\rho})}{|K(x,y)-K(x,z)|^r+|K(y,x)-K(z,x)|^rdx} \right)^{\frac{1}{r}} \nonumber \\
	&\quad \quad \quad \quad \quad \quad \quad \quad \quad \quad \quad \quad \quad \quad \quad \quad \quad \lesssim |C_j(z,\ell^{\rho})|^{\frac{1}{r}-1+\frac{\delta}{n} \left(\frac{1}{\rho}-\frac{1}{\sigma}  \right)}  2^{-\frac{j\delta}{ \rho}}
\end{align}
for $\ell <1$, where $z \in \R^n$, $|y-z|<\ell$, $0< \rho \leq \sigma$ and $C_j(z,\eta):=\{ x \in \R^n: \ 2^j \eta<|x-z| \leq 2^{j+1}\eta \}$.  {Those conditions also covers the standard case $\sigma=1$, by choosing $\rho=\sigma$ in \eqref{1-hormander-2a}, and in that case both conditions are the same.} It is easy to check that $D_{\delta,r_{1}}$ condition is stronger than $D_{\delta,r_{2}}$ for $r_{1} > r_{2}$ and any $\delta>0$ and $0<\sigma\leq1$. Moreover, $\delta-$kernels of type $\sigma$  {satisfying \eqref{pontual_kernel_tipo_sigma} also satisfies $D_{\delta,r}$ condition for all $r \geq 1$. It has also been shown in \cite[Proposition 5.3]{VasconcelosPicon} that pseudodifferential operators associated to symbols in the H\"ormander classes $S^{m}_{\sigma,\mu}(\R^{n})$ with $0<\sigma \leq 1$, $0 \leq \mu <1$, $\mu \leq \sigma$ and $m\leq -n(1-\sigma)/2$, satisfies the $D_{1,r}$ condition for $1 \leq r \leq 2$. We refer to \cite{VasconcelosPicon} for more details.} In particular, the continuity of operators associated to symbols given by $e^{i|\xi|^{\sigma}}|\xi|^{-m}$ away from the origin are also examples of this type of operators and have been extensively studied, for instance in \cite{Abbico-Picon-Ebert, Fefferman1970,Hirschman1959,Wainger1965}. 

Our main result is the following:


\begin{theorem} \label{theorem-CZ-hardy-morrey}
	Let $T$ to be a strongly singular Calder\'on-Zygmund operator associated to a $D_{\delta,r}$ kernel of type $\sigma$ for some $1\leq r \leq 2$. Under the assumptions that $T^{\ast}(x^{\alpha})=0$ for every $|\alpha|\leq \lfloor \delta \rfloor$, $T$ can be extended to a bounded operator from $\mathcal{HM}_{q}^{\lambda}(\R^n)$ to itself for any $0<q\leq \lambda < r$ and $q_{0}<q\leq 1$, where $q_{0}$ is given by \eqref{critical-index}.
\end{theorem}

 {The proof relies on showing that $T$ maps atoms into molecules and a molecular decomposition in $\mathcal{HM}_{q}^{\lambda}(\R^n)$ for $0<q\leq 1$ and $q\leq \lambda <\infty$ under restriction $\lambda<r$ (see Theorem \ref{thm-mol-HM-2} and the Remark \ref{remark3}). }
 {As an immediate consequence of previous theorem, we also obtain the continuity of standard non-convolution Calder\'on-Zygmund operators ($\sigma=1$) associated to kernels satisfying integral conditions. The corresponding result in the convolution setting for kernels satisfying derivative conditions can be found in \cite[Section 2.2]{JiaWang-SingularIntegrals}.}

\begin{corollary}
	Under the same hypothesis of the previous theorem, if $T$ is a standard Calder\'on-Zygmund operator, then it is bounded from $\mathcal{HM}_{q}^{\lambda}(\R^n)$ to itself provided that $n/(n+\delta)<q\leq 1$.
\end{corollary}

The organization of the paper is as follows. In Section \ref{Section:HardyMorreySpaces} we recall some basic definitions and a general atomic and molecular decomposition of Hardy-Morrey spaces. In particular, in Section \ref{Subsection:AtomicDecomposition} we present an atomic decomposition in terms of $L^r-$atoms by showing the equivalence with classical $L^{\infty}$ atomic space and in Section \ref{Subsection:MolecularDecomposition} we show an appropriated molecular decomposition of Hardy-Morrey spaces. Finally, in Section \ref{Section:ProofOfMainTheorem} we present the proof of Theorem \ref{theorem-CZ-hardy-morrey} showing that  $T$ maps atoms into molecules.
\\

\noindent  {\textbf{Notation:} throughout this work, the symbol $f \lesssim g$ means that there exist a constant $C>0$, not depending on $f$ nor $g$, such that $f \leq C \, g$. By a dyadic cube we mean cubes on $\R^n$, open on the right whose vertices are adjacent points of the lattice $(2^{-k}\Z)^n$ for some $k\in \Z$. Given a set $A\subset \R^n$ we denote by $|A|$ its Lebesgue measure. Given a cube $Q$ (dyadic or not), we will always denote its center and side-length by $x_Q$ and $\ell_Q$ respectively. By $Q^{\ast}$ we mean the cube with same center as $Q$ and side-length $2\ell_Q$. We also denote by $\fint_{Q}f(x)dx := \frac{1}{|Q|}\int_{Q}f(x)dx$.}

\section{Hardy-Morrey spaces $\mathcal{HM}_q^\lambda(\R^n)$} \label{Section:HardyMorreySpaces}

In this section, we recall and present some properties of Hardy-Morrey spaces. For $0 < q\leq \lambda<\infty$, the Morrey spaces, denoted by $\mathcal{M}^\lambda_q(\mathbb{R}^n)$, are defined to be the set of measurable functions $f\in L^q_{loc}(\mathbb{R}^n)$ such that 
$$
\displaystyle{	\Vert f \Vert_{\mathcal{M}^\lambda_q}:=\sup_{J}\vert J\vert^{\frac{1}{\lambda}-\frac{1}{q}}\left( \int_J\vert f(y)\vert^qdy \right)^{\frac{1}{q}}<\infty},
$$
where the supremum is taken over all cubes $J\subset \R^n$. 

For any tempered distribution $f\in\mathcal{S}'(\mathbb{R}^n)$ and any fixed $\varphi \in \mathcal{S}(\mathbb{R}^{n})$ with $\int \varphi \neq 0$, consider the smooth maximal function $M_{\varphi}f(x)=\sup_{t>0} \left\vert (\varphi_{t}\ast f)(x)\right\vert $, where $\varphi_{t}(x)=t^{-n}\varphi(x/t)$. For any $0<q\leq \lambda<\infty$, we say that $f\in\mathcal{S}'(\mathbb{R}^{n})$ belongs to Hardy-Morrey space $\mathcal{HM}_q^{\lambda}(\R^n)$ if the smooth maximal function $M_{\varphi}f\in\mathcal{M}_q^{\lambda}(\R^n)$. The functional $\Vert f\Vert_{\mathcal{HM}_q^{\lambda}}:=\Vert M_{\varphi}f\Vert_{\mathcal{M}_q^{\lambda}}$ defines a quasi-norm as $0<q<1$ and is a norm if $q\geq1$. 

In the same way as Hardy spaces, the Hardy-Morrey spaces have also equivalent maximal characterizations (see \cite[Section 2]{JiaWang-HM}).  
Clearly, Hardy-Morrey spaces cover the classical Hardy spaces $H^p(\R^n)$ when $\lambda = q$ and Morrey spaces $\mathcal{M}^\lambda_q(\mathbb{R}^n)$ if $1<q \leq \lambda < \infty$.

\subsection{Atomic decomposition in Hardy-Morrey spaces} \label{Subsection:AtomicDecomposition}




\begin{definition} \cite[Definiton 2.2]{JiaWang-SingularIntegrals}. Let $0<q\leq 1\leq r \leq \infty$ with $q<r$ and $q \leq \lambda<\infty$. A measurable function $a_{Q}$ is called a $(q,\lambda,r)-$atom if it is supported on a cube $Q\subset \R^n$ and satisfies:
	(i)  $\| a_{Q} \|_{L^r} \leq \vert Q\vert ^{\frac{1}{r}-\frac{1}{\lambda}}$  and $(ii) \ \int_{\mathbb{R}^n} x^{\alpha}a_{Q}(x)dx =0$
	for all $\alpha\in\mathbb{N}_{0}^{n}$ such that $\vert\alpha\vert\leq  N_{q}:=\left\lfloor n\left({1}/{q}-1\right)\right\rfloor$, where $\lfloor \cdot \rfloor$ denotes the floor function.  
\end{definition}

The following lemma is an extension of \cite[Proposition 2.5]{MT} and the proof will be presented for completeness.

\begin{proposition}\label{prop1} Let $0<q\leq 1 \leq r \leq \infty$ with $q<r$ and $q\leq\lambda<\infty$ with $\lambda \leq r$. If $f$ is a compactly supported function in $L^{r}(\R^n)$ satisfying the moment condition 
	\begin{align}\label{moment}
		\int_{\R^{n}} x^{\alpha}f(x)dx=0\; \text{ for all }\; |\alpha|\leq  N_{q},
	\end{align}
	then it belongs to $\mathcal{HM}_q^{\lambda}(\R^n)$ and moreover $\|f\|_{\mathcal{HM}_{q}^{\lambda}} \lesssim \|f\|_{L^{r}}|Q|^{1/ \lambda-1/r}$ for all cube $Q\supseteq \supp(f)$. In particular, if $f=a_Q$, then $\| a_Q \|_{\mathcal{HM}_{q}^{\lambda}} \lesssim 1$ uniformly.
\end{proposition}

\begin{proof}
	Let $J \subset \R^n$ be an arbitrary cube and $Q$ a cube such that $\supp(f) \subseteq Q$.  {Split the integral over $J$ into $J \cap Q^{\ast}$ and $J\setminus Q^{\ast}$.} Since the maximal function $M_{\varphi}$ is bounded from $L^r(\R^n)$ to itself for every $1<r\leq \infty$, it follows that 
	\begin{align*}
		\int_{J \cap Q^{*}}\vert M_{\varphi}f(x)\vert^qdx
		& \leq \Vert M_{\varphi}f\Vert_{L^r}^q \,  |J \cap Q^{\ast}|^{1-\frac{q}{r}} 
	 \lesssim \| f \|_{L^r}^q \, |J \cap Q^{\ast}|^{1-\frac{q}{r}}.
	\end{align*}
	For $r=1$ and $0<q<1$, setting $R=\| f\|_{L^1}|J\cap Q^{\ast}|^{-1}$ and using that $M_{\varphi}$ satisfies weak $(1,1)$ inequality we get the analogous inequality:
		\begin{align}
			&\int_{J \cap Q^{*}}\vert M_{\varphi}f(x)\vert^qdx \nonumber 
			\, \simeq  \int_{0}^{\infty} \omega^{q-1} \left| \{ x \in J\cap Q^{\ast}: \, |M_{\varphi}f(x)|>\omega  \} \right|d\omega \nonumber \\
			&\quad \lesssim \, |J\cap Q^{\ast}| \int_{0}^{R} \omega^{q-1}d\omega +  \| f\|_{L^1} \int_{R}^{\infty} \omega^{q-2}d\omega \lesssim \| f \|_{L^1}^q \, |J \cap Q^{\ast}|^{1-{q}}. \label{weak-ineq}
	\end{align}
	If $|Q|<|J|$, since $q/\lambda-1 \leq 0$ and $1-q/r>0$ for all $1\leq r <\infty$, one has $|J|^{{q}/{\lambda}-1} |J \cap Q^{\ast}|^{1-{q}/{r}} \leq |Q|^{1-q/r}$. On the other hand, if $|J|<|Q|$, using that $\lambda \leq r$ it follows
	$ \displaystyle{
	|J|^{\frac{q}{\lambda}-1} |J \cap Q^{\ast}|^{1-\frac{q}{r}} = |J|^{\frac{q}{\lambda}-\frac{q}{r}} \, \left( \frac{|J \cap Q^{\ast}|}{|J|} \right)^{1-\frac{q}{r}} \leq |Q|^{1-\frac{q}{r}}}.
	$
	Hence
	$
	|J|^{\frac{q}{\lambda} -1}\int_{J \cap Q^{*}}\vert M_{\varphi}f(x)\vert^qdx
	\lesssim \| f \|_{L^r}^{q} \, |Q|^{1-\frac{q}{r}}.
	$
	
	To estimate the integral on  $J\backslash Q^{\ast}$, using the moment condition \eqref{moment} we write 
	$
	\varphi_{t}\ast f (x)= \int f(y) \left(\varphi_{t}(x-y)-P_{\varphi_t}(y) \right)dy,
	$
	where $\displaystyle{P_{\varphi_t}(y)=\sum_{\vert \alpha\vert \leq N_q}{C_{\alpha}\, \partial^{\alpha} \varphi_{t}(x) \, (-y)^{\alpha}}}$ denotes the Taylor polynomial of degree $N_{q}$ of the function $y \mapsto \varphi_{t}(x-y)$.
	The standard estimate of the remainder term {(see \cite[p. 106]{Stein})} 
	yields 
	$
	\left| \varphi_{t}(x-y)- P_{\varphi_t}(y) \right| \lesssim {|y-x_Q|^{N_q+1}}{|x-x_Q|^{-(n+N_q+1)}}
	$ 
	and  since $\supp(f)\subseteq Q$, we have the pointwise control 
	\begin{align*}
		\left\vert  M_{\varphi}f(x) \right\vert &\lesssim \frac{\ell_Q^{N_q+1}}{|x-x_Q|^{n+N_q+1}}  \int_{Q}{|f(y)|dy } 
		\lesssim \frac{\ell_Q^{N_q+1}}{|x-x_Q|^{n+N_q+1}} \, \| f \|_{L^r} \, |Q|^{1-\frac{1}{r}}.
	\end{align*}
	If $|Q|<|J|$, since $N_q+1 > n\left(1/q-1\right)$, we estimate $|J|^{\frac{q}{\lambda} -1} \int_{J \setminus Q^{*}}\vert M_{\varphi}f(x)\vert^qdx$ by
	\begin{align*}
		\| f \|_{L^r}^q \, |Q|^{q\left(\frac{1}{\lambda}-\frac{1}{r}+\frac{N_q}{n}+\frac{1}{n}+1\right)-1}  \int_{(Q^{\ast})^c} \vert x-x_Q\vert ^{-q(n+N_q+1)}dx \lesssim \| f \|_{L^r}^q \, |Q|^{\frac{q}{\lambda}-\frac{q}{r}}.
	\end{align*}
	Finally, if $|J|<|Q|$ 
	\begin{align*}
		|J|^{\frac{q}{\lambda} -1}  \int_{J \setminus Q^{*}}\vert M_{\varphi}f(x)\vert^qdx &\lesssim\| f \|_{L^r}^q \, |J|^{\frac{q}{\lambda} -1} \, |Q|^{q-\frac{q}{r}} \, \ell_{Q}^{\, -nq} \,  \ |J \backslash Q^{\ast}| 
		 \lesssim \| f \|_{L^r}^q \, |Q|^{\frac{q}{\lambda}-\frac{q}{r}},
	\end{align*}
	which concludes the proof. 
\end{proof}

Given $1\leq r \leq \infty$, we denote the atomic space $\textbf{at}\mathcal{HM}_{q}^{\lambda,r}(\R^n)$ by the collection of $f\in\mathcal{S}'(\mathbb{R}^n)$ such that
$f =\sum_{Q\,:\, \text{dyadic}}s_{Q}a_{Q} \,\, \text{ in } \,\, {\mathcal{S}'(\mathbb{R}^n)}$,
where $\{a_Q\}_Q$ are  $(q,\lambda,r)$-atoms and $\{s_Q\}_{Q}$ is a sequence of complex scalars satisfying  
\begin{equation}\nonumber
\Vert \{s_Q\}_{Q}\Vert_{{\lambda,q}}:=\sup_{J} \left\{\left({\vert J\vert^{\frac{q}{\lambda}-1}} \sum_{\substack{Q\subseteq J}}\left(\vert Q\vert^{\frac{1}{q}-\frac{1}{\lambda}}\,\vert s_{Q}\vert\right)^q \right)^{\frac{1}{q}}\right\}<\infty.
\end{equation}
The functional $\Vert f \Vert_{\textbf{at}\mathcal{HM}_{q}^{\lambda,r}}:= \inf \left\{  \Vert \{s_Q\}_{Q}\Vert_{{\lambda,q}} : f =\sum_{Q}s_{Q}a_{Q} \right\}$,
where the infimum is taken over all such atomic representations, defines a quasi-norm in $\textbf{at}\mathcal{HM}_{q}^{\lambda,r}(\R^n)$. Clearly, if $1\leq r_{1}<r_{2}\leq \infty$ then $\textbf{at}\mathcal{HM}_{q}^{\lambda,r_{2}}(\R^n)$ is continuously embedded in $\textbf{at}\mathcal{HM}_{q}^{\lambda,r_{1}}(\R^n)$. The converse of this simple embedding is the content of the next result.

\begin{lemma}\label{atomic-lemma} 
	Let $0<q\leq 1 \leq r$ with $q<r$ and $q\leq\lambda<\infty$. Then $\textbf{at}\mathcal{HM}_{q}^{\lambda,r}(\R^n) = \textbf{at}\mathcal{HM}_{q}^{\lambda,\infty}(\R^n)$ with comparable quasi-norms.
\end{lemma}

\begin{proof} The proof is based on the  {corresponding} theorem for Hardy spaces (see \cite[Theorem 4.10]{CF}).  {Let $a_Q$ to be a $(q, \lambda, r)-$atom and we show that $a_Q = \sum_j s_{Q_j} a_{Q_{j}}$, where $\{a_{Q_{j}}\}_{j}$ are $(q, \lambda, \infty)-$atoms and $\Vert\{ s_{Q_{j}}\}_{j}\Vert_{q,\lambda} \leq C$ independently}. Consider $b_Q=|Q|^{1/\lambda} \, a_Q $ and since $\int_Q|b_Q(x)|^rdx\leq \vert Q\vert $, from Calder\'on-Zygmund decomposition applied for $\vert b_Q\vert^r \in L^{1}(Q)$ at level $\alpha^{r}>0$, there exists a sequence $\{ Q_j\}_{j}$  of disjoint dyadic cubes (subcubes of $Q$) such that $\vert b_{Q}(x)\vert \leq \alpha, \ \forall \, x \notin \bigcup_{j}Q_{j}$, $\alpha^{r}\leq \fint_{Q_j}{|b_{Q}(x)|^{r}dx} \leq 2^{n} \alpha^r$ and $\big\vert \bigcup_{j}Q_{j}\big\vert  \leq \alpha^{-r}  \int_Q|b_Q(x)|^rdx \leq  \vert Q\vert \, \alpha^{-r} $. Let $\mathcal{P}_{N_{q}}$ to be the space of polynomials in $\R^{n}$ with degree at most $N_q$ and $\mathcal{P}_{N_q,j}$ its restriction to $Q_j$. Since $\mathcal{P}_{N_q,j}$ is a subspace of the Hilbert space $L^{2}(Q_{j})$, let $P_{Q_{j}}b \mathcal{P}_{N_q,j}$ to be the unique polynomial such that $\int_{Q_{j}}[b_{Q}(x)-P_{Q_{j}}(b)(x)] x^{\beta}dx=0$ for all $|\beta| \leq N_q$. 
	
Now we write $b_Q=g_0+\sum_j h_j$, where $h_j(x)=  [b_Q(x)-P_{Q_j}(b)(x)]\mathds{1}_{Q_j}(x)$ and $g_0(x) = b_Q(x)$ if $x \notin \bigcup_{j}Q_j$ and $g_0(x) =P_{Q_j}(b)(x)$ if $x \in Q_j$. Clearly $\int h_j(x)x^{\beta}dx=0$ and since $|g_0(x)| \leq c \alpha\,$ almost everywhere (see \cite[Remark 2.1.4 p. 104]{Stein}), this implies
\begin{align*}
\left( \fint_{Q_j}{|h_j(x)|^{r}dx} \right)^{{1}/{r}} \leq \left( \fint_{Q_j}{|b_Q(x)|^{r}dx} \right)^{{1}/{r}} + \left( \fint_{Q_j}{|g_{0}(x)|^{r}dx} \right)^{{1}/{r}}  \leq  c \alpha.
\end{align*} 
For each $j_0 \in \N$, let $b_{j_0}(x):=(c \alpha)^{-1}h_{j_0}(x)$ and write $b_Q(x)=g_0(x)+(c\alpha) \sum_{j_0}{b_{j_0}(x)}$, where $\int_{Q_{j_0}}\vert b_{j_0}(x)\vert^rdx \leq  |Q_{j_0}|$. Applying the previous argument for each $b_{j_0}$ we obtain the identity 
\begin{align*}
	\displaystyle{b_Q=g_0+(c\alpha)\sum_{j_0}{b_{j_0}}= g_0+c\alpha \sum_{j_0}g_{j_0}+(c \alpha)^{2}\sum_{j_0,j_1}b_{j_0,j_1}},
\end{align*}
where $\int_{Q_{j_{0},j_{1}}}\vert b_{j_{0},j_{1}}(x)\vert^rdx \leq  |Q_{j_{0},j_{1}}|$ and $\{ Q_{j_{0},j_1}\}_{j_1}$ is a sequence of disjoint dyadic cubes (subcubes of $Q_{j_{0}}$) such that $\vert g_{j_{0}}(x)\vert \leq c \alpha$ a.e., $\displaystyle{ \alpha^{r}\leq \fint_{Q_{j_{0},j_1}}{|b_{j_{0}}(x)|^{r}dx} \leq 2^{n} \alpha^r }$ and $ \displaystyle{\big|\bigcup_{j_1}Q_{j_{0},j_1}\big| \leq c\,\alpha^{-r}  \int_{Q_{j_0}}|b_{j_{0}}(x)|^rdx \leq  c \vert Q_{j_{0}}\vert \, \alpha^{-r}}$. Employing an induction argument, we can find a family $\{Q_{i_{k-1},j}\}_{j}:=\{Q_{j_0,\cdots, j_{k-1},j}\}_{j}$ of disjoint dyadic subcubes of  $Q_{i_{k-1}}:=Q_{j_0,\cdots, j_{k-1}}$ for $k=1,2,\cdots$ with $i_{k-1}= \left\{j_0,j_1,\cdots, j_{k-1}\right\}$
such that 
\begin{align}\label{soma}
	b_Q
	&=g_{i_0}+c\alpha \sum_{i_1}g_{i_1}+(c \alpha)^{2}\sum_{i_2}g_{i_2}+\cdots + (c\alpha)^{k-1}\sum_{i_{k-1}}g_{i_{k-1}}+(c\alpha)^{k}\sum_{i_{k}}{h_{i_{k}}},
\end{align}
in which $g_{i_{k-1}}$ and $h_{i_{k}}$, for every $i_{k}=(j_0,j_1,\cdots, j_{k-1},j)$, satisfies 
$\vert g_{i_{k-1}}(x)\vert \leq c \alpha$ a.e. $x \in \R^n$, $\displaystyle{\alpha^{r}\leq \fint_{Q_{i_{k-1},j}}{|h_{i_k}(x)|^{r}dx} \leq 2^{n} \alpha^r}$ and $\displaystyle{\big|\bigcup_{j} Q_{i_{k-1},j}\big|\leq   c \vert Q_{i_{k-1}}\vert \, \alpha^{-r}}$. 
The sum at \eqref{soma} is interpreted as $ \sum_{i_{k-1}}g_{i_{k-1}}:=\sum_{{j_0}\in\mathbb{N}} \cdots \sum_{j_{k-1}\in\mathbb{N}}g_{j_0,\cdots,j_{k-1}}$ (analogously to $\sum_{i_{k}}h_{i_{k}}$). We claim that the reminder term  $(c\alpha)^{k}\sum_{i_{k}}h_{i_{k}}$ in \eqref{soma} goes to zero in $L^1(\R^n)$ as $k\rightarrow \infty$. Indeed, writing $Q_{i_{k}}:=Q_{i_{k-1},j}$ for some fixed $j$ we have   
$	\int_{\R^n}{|h_{i_k}(x)|dx}=\int_{Q_{i_k}}{|h_{i_k}(x)|dx} \leq \Big( \int_{Q_{i_k}}{|h_{i_k}(x)|^rdx} \Big)^{\frac{1}{r}} \, |Q_{i_k}|^{1-\frac{1}{r}}
	\leq c \alpha |Q_{i_k}|$
and iterating $(k+1)$-times the previous argument one has
\begin{equation}\label{k-times-iii}
	\sum_{i_{k}}{|Q_{i_{k}}|} \leq \left( \frac{c}{\alpha^r}\right)^{k+1}|Q|.
\end{equation} 
Thus, ${\int{ \left| (c\alpha)^{k} \sum_{i_k}h_{i_k}(x)\right| dx} \leq  (c \alpha)^{k+1} \sum_{i_k}{|Q_{i_k}|} \leq {(c^{2}\alpha^{1-r})}^{(k+1)} |Q|}$. That means, $(c\alpha)^{k}\sum_{i_{k}}h_{i_{k}}(x)$ goes to $0$ in $L^1(\R^n)$ as $k\rightarrow\infty$, provided that $c^{2}\alpha^{1-r}<1$. Therefore,  
\begin{align*}
	b_Q=g_{i_{0}}+c\alpha \sum_{i_1}g_{i_1}+(c \alpha)^{2}\sum_{i_2}g_{i_2}+\cdots + (c\alpha)^{k-1}\sum_{i_{k-1}}g_{i_{k-1}}+(c\alpha)^{k}\sum_{i_{k}}g_{i_{k}}+\cdots\; 
\end{align*}
in $L^1(\R^n)$, where  $|g_{i_{k}}(x)|\leq c\alpha$ a.e. and for all $|\beta|\leq N_{q}$ we have
$\int{x^{\beta}g_{i_{k}}(x)dx} = \int{x^{\beta}b_{i_{k}}(x)dx}+\sum_{j}\int_{Q_{i_{k-1},j}}{x^{\beta}P_{Q_{i_k,j}}b(x)dx} = \int x^{\beta}b_{i_k}(x)dx=0 $. From the above considerations it is clear that  $a_{i_0}:=(c\alpha)^{-1} \, |Q|^{-{1}/{\lambda}} \, g_{i_0}$ and $a_{i_{k}}:=(c\alpha)^{-1} \, |Q_{i_{k}}|^{-{1}/{\lambda}} \, g_{i_{k}}$ are  $(q,\lambda, \infty)-$atoms, for all $k=1,2,\cdots$. Moreover, we can write  
\begin{align}
	a_Q 
	&= s_{i_{0}} a_{i_0} + \sum_{i_1}s_{i_1} a_{i_1} + \sum_{i_2}s_{i_2} a_{i_2}
	+\cdots +\sum_{i_{k}}s_{i_{k}}a_{i_{k}}+\cdots
	\label{key-decAtom}
\end{align}
where each coefficient $\{s_{i_k}\}$ is defined by $s_{i_k}=(c\alpha)^{k+1}\vert Q\vert^{-{1}/{\lambda}}\vert Q_{i_{k}}\vert^{{1}/{\lambda}}$. It remains  to show that $\Vert \{s_{{i_k}}\}_{k}\Vert_{\lambda,q} \leq C$, uniformly. Fixed $J\subset \R^n$ a dyadic cube, we may estimate 
\begin{align*}
	|J|^{\frac{q}{\lambda} -1}\sum_{k=0}^{\infty}	\sum_{Q_{i_k}\subseteq J} \vert s_{i_{k}}\vert^q \vert Q_{i_{k}}\vert^{1-\frac{q}{\lambda}} &= |J|^{\frac{q}{\lambda}-1} \, |Q|^{-\frac{q}{\lambda}}	\sum_{k=0}^{\infty}(c \alpha)^{q(k+1)} \Big(\sum_{\substack{Q_{i_k}\subseteq J}}\vert Q_{i_{k}}\vert\Big) \\
	&\lesssim \vert J\vert^{\frac{q}{\lambda}-1} |Q|^{-\frac{q}{\lambda}}|J \cap Q| \sum_{k=0}^{\infty}(c \alpha)^{q(k+1)}\left(\frac{c}{\alpha^{r}}\right)^{k+1} \leq C
\end{align*}
provided $c^{q+1}\alpha^{q-r}<1$ (weaker than the previous one) and $q \leq \lambda$. Note that here we have used a refinement of \eqref{k-times-iii} given by 
$	\sum_{i_{k} \, : \, Q_{i_k}\subseteq J }{|Q_{i_{k}}|} \lesssim \left( \frac{c}{\alpha^r}\right)^{k+1}|J \cap Q|$
and the uniform control $\vert J\vert^{q/\lambda-1} |Q|^{-q/\lambda}|J \cap Q| \lesssim 1$.  
\,
\end{proof}

The previous lemma allow us to study Hardy-Morrey spaces $\mathcal{HM}_q^{\lambda}(\R^n)$ with any of the atomic spaces $\textbf{at}\mathcal{HM}_{q}^{\lambda,r}(\R^n)$ for $1 \leq r \leq \infty$ provided $q<r$. In addition, we announce an atomic decomposition in terms of $(q,\lambda,r)-$atoms,  {which is a direct consequence of the one proved in \cite[p. 100]{JiaWang-HM} for $(q,\lambda,\infty)-$atoms and the Lemma \ref{atomic-lemma}, since they are in particular $(q,\lambda,r)-$atoms.}

\begin{theorem} \label{theorem-atomic-decomposition}
	{Let $0< q \leq 1 \leq r \leq \infty$ with $q<r$ and $q\leq\lambda<\infty$. Then, $f\in \mathcal{HM}_q^{\lambda}(\R^{n})$ if and only if there exist a collection of $(q,\lambda,r)-$atoms $\{a_Q\}_Q$ and a sequence of complex numbers $\{s_{Q}\}_Q$  such that $f =\sum_{Q}s_{Q}a_{Q}$ in $ \mathcal{S}'(\R^n)$  and $ \Vert f \Vert_{\textbf{at}\mathcal{HM}_{q}^{\lambda}} \approx \Vert f\Vert_{\mathcal{HM}_q^\lambda}$.}
\end{theorem}

\subsection{Molecular decomposition in Hardy-Morrey spaces} \label{Subsection:MolecularDecomposition}



\begin{definition} \label{molecules} Let $0<q \leq 1 \leq r<\infty$ with $q<r$, $q \leq \lambda<\infty$, and $s>n\left({r}/{q}-1\right)$. A function $m(x)$ is called a  $(q,\lambda, {s}, r)-$molecule  in $\mathcal{HM}_q^{\lambda}(\R^{n})$ {, or simply an $L^{r}-$molecule},  if there exist a cube $Q$ such that
	$$
	(M_1) \ \int_{\R^{n}}{|m(x)|^{r}dx} \lesssim \, \ell_{Q}^{\ n\left(1-\frac{r}{\lambda}\right)} \quad (M_2) \ \int_{\R^{n}} |m(x)|^r|x-x_Q|^s dx \lesssim \, \ell_{Q}^{\ s + n\left(1-\frac{r}{\lambda} \right)}
	$$
	and also the concelation condition $\displaystyle (M_3) \ \int_{\R^{n}} {m(x)x^{\alpha}dx}=0$ for all $|\alpha| \leq N_{q}$. 
	
\end{definition}


\begin{remark} \label{remark-molecules}
	\textnormal{Equivalently, we can replace the previous global estimates by $(M_1)$ on $2Q$ and $(M_2)$ on $2Q^c$.}
\end{remark}

\begin{lemma}\label{mol-decomp-generalized}
	Let $m(x)$ to be an  { $L^{r}-$molecule}. Then
	$	m=\sum_Q d_Q \, a_Q+\sum _Q t_Q \, b_Q$  in $L^r(\mathbb{R}^n)$,
	where each  $\{a_Q\}_{_Q}$ are $(q,\lambda,r)-$atoms and $\{b_Q\}_{_Q}$ are $(q,\lambda,\infty)-$atoms, for a suitable sequence of scalars $\{d_Q\}_{Q}$ and  $\{t_Q\}_{Q}$.
\end{lemma}

\begin{proof}
The proof follows the corresponding result for for Hardy spaces \cite[Theorem 7.16]{CF}. Let $m$ to be a $(q,\lambda, {s}, r)-$molecule centered in the cube $Q$. For each $j\in\mathbb{N}$, let $Q_{j} := Q(x_Q,\ell_{j})$ in which $\ell_{j}=2^j\ell_{Q}$. Consider the collection of annulus $\{E_j\}_{j\in\mathbb{N}_0}$ given by $E_{0}=Q$ and $E_j=Q_{j}\backslash Q_{j-1}$ for $j \geq 1$, and let  $m_{j}(x):=m(x)\, \mathds{1}_{{E_j}}(x)$.  {By the same arguments presented \textcolor{blue}{in} the proof of Lemma \ref{atomic-lemma}}, there exist polynomials $\{\phi^j_{\gamma}(x)\}_{\vert\gamma\vert\leq N_q}$ uniquely determined in $E_j$ such that
\begin{align}\label{cancel-delta}
(2^j\ell_Q)^{|\gamma|} |\phi_{\gamma}^j(x)|\lesssim 1 \quad \text{ and }\quad \frac{1}{|E_j|}\int_{E_j}\phi^j_{\gamma}(x) x^{\beta}dx=\begin{cases}
1, &\gamma=\beta\\
0, &\gamma\neq\beta
\end{cases}
\end{align}
where the implicit constant is uniformly on $E_j$. Let ${m_{\gamma}^j =\fint_{E_j} m_{j}(x)x^{\gamma}dx}$ and consider ${P_{j}(x)=\sum_{|\gamma|\leq N_q}m_{\gamma}^j\phi_{\gamma}^{j}(x)}$. Splitting $m=\sum_{j=0}^{\infty}\left(m_{j}-P_{j}\right)+\sum_{j=0}^{\infty}P_{j}$,
with convergence in $L^r(\R^n)$, we claim that for each $j$, $m_{j}-P_{j}$ is multiple of a $(q,\lambda,r)$-atom and $P_{j}$ is a finite linear combination of $(q,\lambda,\infty)$-atoms.

For the first sum, since $m_{j}$ and $P_{j}$ are supported on $E_j$, so is $m_{j}-P_{j}$ and by definition one has the desired vanish moments up to the order $N_q$. It remains show that $m_{j}-P_{j}$ satisfies the size estimate. Indeed, from conditions $(M_1)$ and $(M_2)$ it follows that for every $j\in \N_{0}$
\begin{equation}
\| m_{j} \|_{L^r} \lesssim |E_j|^{ \frac{1}{r}-\frac{1}{\lambda} } \ (2^{j})^{-\frac{s}{r}+ n\left(\frac{1}{\lambda}-\frac{1}{r} \right)}. \label{estimation-mj}
\end{equation}
Also, from \eqref{cancel-delta} it follows 
$\vert P_{j}(x)\vert\leq \Big(\sum_{|\beta|\leq N_q} \vert \phi_{\beta}^{j}(x)\vert 2^{j|\beta|}\Big) \fint_{E_j} \vert m_{j}(x)\vert dx \lesssim \vert E_j\vert ^{-\frac{1}{r}}\Vert m_{j}\Vert_{L^r}$, 
where the implicit constants are independent of $j$. Hence, if we write $(m_{j}-P_{j})(x)=d_{j} \, a_{Q_{j}}(x)$ for 
$d_{j}=\Vert m_{j}-P_{j}\Vert_{L^r} \, \vert Q_j\vert^{\frac{1}{\lambda}-\frac{1}{r}}$ and $a_{Q_{j}}=\frac{m_{j}-P_{j} }{\Vert m_{j}-P_{j}\Vert_{L^r}} \ \vert Q_j\vert^{\frac{1}{r}-\frac{1}{\lambda}}$,
for each $j\in\mathbb{N}_{0}$, it is clear that $\{a_{Q_j}\}_{j}$ is a sequence of $(q,\lambda,r)$-atoms supported on $Q_j$. Moreover, from (\ref{estimation-mj}) we have 
$
\| m_{j} - P_{j} \|_{L^r} \lesssim \| m_{j} \|_{L^r} \lesssim |Q_j|^{ \frac{1}{r}-\frac{1}{\lambda} } \ (2^{j})^{-\frac{s}{r}+ n\left(\frac{1}{\lambda}-\frac{1}{r} \right)}.
$ 
Hence, since $s>n\left({r}/{q}-1\right)$
\begin{align*}
	\sum_{j=0}^{\infty} \vert d_{j}\vert^q \vert Q_j\vert ^{1-\frac{q}{\lambda}} & \lesssim |Q|^{1-\frac{q}{\lambda}} \sum_{j=0}^{\infty} (2^j)^{q\left[-\frac{s}{r}+n\left( \frac{1}{q}-\frac{1}{r}\right) \right]} 
	\lesssim |Q|^{1-\frac{q}{\lambda}}.
\end{align*}

For the second sum, let $\psi_{\gamma}^{j}(x):=N_{\gamma}^{j+1} \left[ |E_{j+1}|^{-1} \phi_{\gamma}^{j+1}(x)-|E_j|^{-1}\phi_{\gamma}^{j}(x) \right]$, where
	$ N_{\gamma}^{j}=\sum_{k=j}^{\infty}{m^{k}_{\gamma}|E_k|}$ $= \sum_{k=j}^{\infty}{\int_{E_k}{m_Q(x)x^{\gamma}dx}} $.
Then, we can represent $P_{j}$ (using the vanish moments ($M_3$)) as
$\displaystyle{\sum_{j=0}^{\infty}P_{j}(x) = \sum_{j=0}^{\infty}{\sum_{|\gamma|\leq N_q}{\psi_{\gamma}^{j}(x)}}}$. 
%
The function $\psi_{\alpha}^j$ is supported on $E_{j+1}$ and by construction also satisfies vanish moments conditions up to the order $N_q$. It remain to check the size condition.
	Since $|\gamma|\leq n \left( 1/\lambda-1 \right)$ and $s> n \left( r/q-1 \right)$ we have 
		$|N_{\gamma}^{j+1}|	 \leq |Q_j|^{1-{1}/{\lambda}} \  (2^j\ell_Q)^{|\gamma|} (2^j)^{-{s}/{r} +n \left( {1}/{\lambda}-{1}/{r} \right)}$. 
The previous estimate and $(2^j\ell_Q)^{|\gamma|} |\phi_{\gamma}^j(x)|\leq C$ yields for all $x\in E_j$
\begin{equation*}\label{key-est2}
	\left| N^{j+1}_{\gamma}|E_j|^{-1}\phi_{\gamma}^{j}(x) \right| \leq  C  |Q_j|^{-\frac{1}{\lambda}}(2^j)^{-\frac{s}{r} +n \left( \frac{1}{\lambda}-\frac{1}{r} \right)}.
\end{equation*}
Let $\psi_{\gamma}^j=t_{j}\,b_{\gamma}^j$, where $t_{j}=(2^{j})^{-{s}/{r} +n \left( {1}/{\lambda}-{1}/{r} \right)}$ and $b_{\gamma}^{j}(x)=(2^{j})^{{s}/{r} -n \left( {1}/{\lambda}-{1}/{r} \right)}  \ \psi_{\gamma}^j(x)$.
Hence, we can write 
	$\sum_{j=0}^{\infty}P_{j}(x)  = \sum_{j=0}^{\infty}{\sum_{|\gamma|\leq N_q}}t_{j}\,b_{\gamma}^j(x)$,
and for each $j\in\mathbb{N}$ the function $b_{\gamma}^j(x)$ is a $(q,\lambda,\infty)-$atom,  since is  supported on $E_{j+1}$ and satisfies $\vert b_{\gamma}^j(x)\vert \lesssim  |Q_j|^{-\frac{1}{\lambda}}$, as desired. Moreover from $s>n\left({r}/{q}-1\right)$ one has 
\begin{align*}
	\sum_{j=0}^{\infty}|t_{j}|^q |Q_j|^{1-\frac{q}{\lambda}} = |Q|^{1-\frac{q}{\lambda}} \sum_{j=0}^{\infty} (2^j)^{q\left(-\frac{s}{r}+n\left(\frac{1}{q}-\frac{1}{r}\right)\right)} \lesssim |Q|^{1-\frac{q}{\lambda}}.
\end{align*}
\,
\end{proof}

Now we ready to announce a molecular decomposition in Hardy-Morrey spaces.

\begin{theorem}\label{thm-mol-HM-2} 
Let $\left\{ m_{Q}\right\}_{Q}$ be a collection of   {$L^{r}-$molecules} and $\left\{s_{Q}\right\}_{Q}$ be a sequence of complex numbers such that $	\Vert \{s_{Q} \}_{Q} \Vert_{\lambda,q} < \infty$. If the series $f=\sum_{Q}s_{Q}m_{Q}$ converges in $\mathcal{S}'(\R^n)$  {and $\lambda<r$}, then $f \in \mathcal{HM}_{q}^{\lambda}(\R^{n})$ and moreover, $\|f\|_{\mathcal{HM}_{q}^{\lambda}} \lesssim \Vert \{s_{Q}\}_{Q} \Vert_{\lambda,q}$ with implicit constant independent of $f$. 
\end{theorem}

\begin{proof}
	Suppose $f=\sum_{Q}s_{Q}m_{Q}$ in $\mathcal{S}'(\R^n)$ and $\| \{s_Q\}_{Q} \|_{\lambda,q}<\infty$. Since $0<q\leq 1$, for a fixed dyadic cube  $J \subset \R^n$ we may estimate  $\int_{J}|M_{\varphi}f(x)|^{q}dx$ by
	\begin{align*}
		\sum_{Q \subseteq J}|s_{Q}|^{q} \int_{J}|M_{\varphi}m_{Q}(x)|^{q}dx + \sum_{J \subset Q}|s_{Q}|^{q} \int_{J}|M_{\varphi}m_{Q}(x)|^{q}dx 
		& := I_{1}+I_{2}.
	\end{align*}
	\noindent \textbf{Estimate of $I_{1}$}.	
		From Lemma \ref{mol-decomp-generalized}, write $m_{Q}=\sum_{j=0}^{\infty}d_{_j}a_{Q_j}$ (convergence in $L^{r}$)
		where $\left\{a_{Q_j}\right\}_{j}$ are $(q,\lambda,r)-$atoms and moreover $	\sum_{j=0}^{\infty} \vert d_{j}\vert^q \vert Q_j\vert ^{1-\frac{q}{\lambda}}  \lesssim |Q|^{1-\frac{q}{\lambda}}$.
		It follows from analogous estimates of Proposition \ref{prop1} that 
		\begin{align*}
			I_{1}  &\lesssim \sum_{Q \subseteq J}|s_{Q}|^{q} \sum_{j=0}^{\infty}|d_{Q_{j}}|^{q}\int_{J}|M_{\varphi}a_{Q_{j}}(x)|^{q}dx  \lesssim \sum_{Q \subseteq J}|s_{Q}|^{q}  \sum_{j=0}^{\infty}|d_{Q_{j}}|^{q}|Q_{j}|^{1-\frac{q}{\lambda}} \\
			& \lesssim \sum_{Q \subseteq J}|s_{Q}|^{q} |Q|^{1-\frac{q}{\lambda}} \lesssim |J|^{1-\frac{q}{\lambda}} \, \|\{s_{Q}\}_{Q}\|_{\lambda,q}^q.
		\end{align*}
		\noindent \textbf{Estimate of $I_{2}$.} Since $1< r < \infty$ and 
		$M_{\varphi}$ is bounded on $L^r(\R^n)$, it follows
			\begin{align*}
				|J|^{\frac{q}{\lambda}-1} \int_{J}{\vert M_{\varphi}m_Q(x)\vert ^{q}dx} 
				&\leq |J|^{q\left(\frac{1}{\lambda}-\frac{1}{r}\right)} \left( \sum_{j=0}^{\infty} (2^j\ell_Q)^{-s} \int_{E_j}{|m_Q(x)|^r|x-x_Q|^{s}dx} \right)^{\frac{q}{r}} \\
				&\lesssim |J|^{q\left(\frac{1}{\lambda}-\frac{1}{r}\right)} \, |Q|^{q \left(\frac{1}{r}-\frac{1}{\lambda}\right)} \left( \sum_{j=0}^{\infty} 2^{-js}  \right)^{\frac{q}{r}} \simeq  \left( \frac{|J|}{|Q|} \right)^{q \left(\frac{1}{\lambda}-\frac{1}{r}\right)}.
			\end{align*}
{If $r=1$ and $0<q<1$, we proceed like in \eqref{weak-ineq} and then $	|J|^{\frac{q}{\lambda} -1}\int_{J}{\vert M_{\varphi}m_Q(x)\vert ^{q}dx} \lesssim$
		\begin{align}
		|J|^{\frac{q}{\lambda} -1} \left[|J| \int_{0}^{|Q|^{1-\frac{1}{\lambda}}|J|^{-1}} \omega^{q-1}d\omega + |Q|^{-1+\frac{1}{\lambda}} \int_{|Q|^{1-\frac{1}{\lambda}}|J|^{-1}}^{\infty} \omega^{q-2}d\omega \right] \lesssim \left( \frac{|J|}{|Q|} \right)^{q \left(\frac{1}{\lambda}-1\right)}. \nonumber 
	\end{align}}
	 Fixed a dyadic cube $J$, we point out there exists a subset $N\subseteq \N$ such that  each cube $J \subset Q$ is uniquely determined by a dyadic cube $Q_{k,J} \in \big\{ Q\; \text{dyadic}: J\subset Q \text{ and }\ell_{Q}=2^{k}\ell_{J} \big\}$. Hence, we can write
$\displaystyle{ \sum_{J \subset Q} |s_{Q}|^{q} \left(\frac{|J|}{|Q|}\right)^{\gamma}=\sum_{k \in N} \vert s_{Q_{k,J}}\vert^q \, 2^{-kn\gamma}}$
with $\gamma:={1}/{\lambda}-{1}/{r}>0$. Then, 
\begin{align}
|J|^{\frac{q}{\lambda}-1}\sum_{J \subset Q}|s_{Q}|^{q}\Vert M_{\varphi}(a_Q)\Vert_{L^q(J)}^q  
&\lesssim \sum_{k \in N}\left( |s_{Q_{k,J}}|^{q}  |Q_{k,J}|^{1-\frac{q}{\lambda}} \right) |Q_{k,J}|^{\frac{q}{\lambda}-1}2^{-kn\gamma} \nonumber \\
&\leq \sum_{k \in N} \left( \sum_{Q \subseteq Q_{k,J}} |s_{Q}|^{q}  |Q|^{1-\frac{q}{\lambda}} \right) |Q_{k,J}|^{\frac{q}{\lambda}-1}2^{-kn\gamma} \nonumber \\
&\lesssim \| \left\{s_{Q} \right\}_{Q}\|_{\lambda,q}^q\sum_{k \in N}2^{-kn\gamma} \nonumber
 \lesssim \| \left\{s_{Q} \right\}_{Q}\|_{\lambda,q}^q. \nonumber
\end{align} 
\,
\end{proof}

 {
\begin{remark}\label{remark3}
The Theorem \ref{thm-mol-HM-2} covers \cite[Theorem 2.6]{JiaWang-SingularIntegrals} when $r=2$ where the natural restriction $\lambda<2$ was omitted. 
\end{remark}
}

\section{Proof of Theorem \ref{theorem-CZ-hardy-morrey}} \label{Section:ProofOfMainTheorem}

\begin{proof} 
Let $a$ be a $(q,\lambda,r)-$atom supported in the cube $Q$. From Theorem \ref{thm-mol-HM-2}, it suffices to show that $Ta$ is a $(q,\lambda,s,r)-$ molecule associated to $Q$. Suppose first that $\ell_Q \geq 1$. Since $T$ is bounded in $L^2(\R^n)$ to itself and $1 \leq r \leq 2$, condition $(M_1)$ follows by 
	\begin{align} \label{M1-control}
	\int_{2Q} |Ta(x)|^rdx \leq |2Q|^{1-\frac{r}{2}} \, \| Ta \|_{L^2}^{r} \lesssim |Q|^{1-\frac{r}{2}} \| a \|_{L^2}^{r} \lesssim |Q|^{1-\frac{r}{\lambda}} \simeq \ell_{Q}^{\ n \left(1-\frac{r}{\lambda}\right)}.
	\end{align}
	For $(M_2)$ using the moment condition of the atom $a$, Minkowski inequality and \eqref{Ls-hormander}, we estimate $\int_{2Q^c}|Ta(x)|^{r}|x-x_Q|^{s}dx$ by
\begin{align*}
		& \sum_{j=1}^{\infty}{\int_{C_j(x_Q,\ell_Q)}{\left| \int_{Q}{[K(x,y)-K(x,x_Q)]a(y)dy} \right|^{r} \, |x-x_Q|^{s}dx}} \\
		& \leq  \sum_{j=1}^{\infty}{ (2^j\ell_Q)^{s} \left\{ \int_{Q}{|a(y)| \left[ \int_{C_j(x_Q,\ell_Q)}{|K(x,y)-K(x,x_Q)|^{r} dx} \right]^{\frac{1}{r}} dy} \right\}^{r}} \nonumber \\
		& \lesssim  \sum_{j=1}^{\infty}{ (2^j \ell_Q)^{\,s -n(r-1)} \,\, 2^{-j r\delta}} \,\, \ell_Q^{\, rn\left( 1-\frac{1}{\lambda} \right)} 
		 \simeq \, \ell_Q^{\, r + n\left( 1-\frac{r}{p} \right)} \sum_{j=1}^{\infty}{2^{j\left[s-n(r-1)-r\delta \right]}} \simeq \, \ell_Q^{\, s + n\left( 1-\frac{r}{p} \right)}, \nonumber 
	\end{align*}
	assuming $s<n(r-1)+r\delta$.  {We remark that for the case $r=1$, one needs to consider $(q,\lambda,s,1)-$molecules and hence $0<q\leq \lambda<1$.} Suppose now that $\ell_Q<1$. Since $T$ is a bounded operator from $L^p(\R^n)$ to $L^2(\R^n)$ and $1< r \leq 2$, condition $(M_1)$ follows by
	$$
	\int_{2Q} |Ta(x)|^rdx \leq |2Q|^{1-\frac{r}{2}} \, \| Ta \|_{L^2}^{r} \lesssim |Q|^{1-\frac{r}{2}} \| a \|_{L^p}^{r} \lesssim |Q|^{1-\frac{r}{\lambda}+r\left(\frac{1}{p}-\frac{1}{2}  \right)} \lesssim |Q|^{1-\frac{r}{\lambda}}.
	$$
To estimate the global $(M_2)$ condition, we consider $0<\rho \leq \sigma \leq 1$ a parameter that will be chosen conveniently latter, denote by $2Q^{\rho} := Q(x_Q,2\ell_{Q}^{\rho})$ and split the integral over $\R^n$ into $2Q^{\rho}$ and $(2Q^{\rho})^{c}$.
	For $2Q^{\rho}$ we use the boundedness from $L^{p}(\R^n)$ to $L^{2}(\R^n)$ again and obtain
	\begin{align}
		\int_{2Q^{\rho}} |Ta(x)|^r \, |x-x_Q|^{s}dx &\lesssim \ell_{Q}^{\, s \rho} |4Q^{\rho}|^{1-\frac{r}{2}} \| Ta \|_{L^2}^{r} \lesssim \ell_{Q}^{\, \rho s +n\rho \left(1-\frac{r}{2} \right)} \| a \|_{L^p}^{r} \nonumber \\
		& \lesssim \ell_{Q}^{\, \rho s +n \left[\rho-\frac{r\rho}{2}+r \left( \frac{1}{p}-\frac{1}{\lambda} \right) \right]} 
		 \lesssim \ell_{Q}^{\, s+n\left(1-\frac{r}{\lambda} \right)}, \nonumber
	\end{align}
	assuming $s \leq -n \left( 1-\frac{r}{2} \right)+\frac{nr}{1-\rho}\left( \frac{1}{p}-\frac{1}{2} \right)$. For $(2Q^{\rho})^{c}$, we use \eqref{1-hormander-2a} to estimate $\int_{(2Q^{\rho})^c}|Ta(x)|^{r}|x-x_Q|^{s}dx $ by
	\begin{align} \label{choose-rho}
		 & \sum_{j=1}^{\infty}{ (2^j\ell_{Q}^{\, \rho})^{s} \left\{ \int_{Q}{|a(y)| \left[ \int_{C_j(x_Q,\ell_{Q}^{\, \rho})}{|K(x,y)-K(x,x_Q)|^{r} dx} \right]^{\frac{1}{r}} dy} \right\}^{r}} \nonumber \\
		& \lesssim \ \sum_{j=1}^{\infty}{(2^j\ell_{Q}^{\, \rho})^{s} \,\, \left(|C_j(x_Q,\ell_{Q}^{\, \rho})|^{{\frac{1}{r}-1+\frac{\delta}{n} \left(\frac{1}{\rho}-\frac{1}{\sigma} \right)}} \,\, 2^{-\frac{j\delta}{\rho}} \right)^{r} \,\, \ell_{Q}^{\, rn\left( 1-\frac{1}{\lambda} \right) }} \nonumber \\
		& \simeq \,  \ell_{Q}^{\, \rho s + n \left[ r+\frac{r\delta}{n}-r\rho \left(1-\frac{1}{r}+\frac{\delta}{n\sigma}\right) -\frac{r}{\lambda} \right]} \sum_{j=1}^{\infty}{2^{j\left[ s-n(r-1)-\frac{r\delta}{\sigma} \right]}} \nonumber \\
		& \lesssim \ \ell_{Q}^{\, \rho s + n \left[ \rho \left(1-\frac{r}{2} \right)+r \left(\frac{1}{p}-\frac{1}{\lambda}\right) \right]} \leq \ell_{Q}^{\, s+n\left(1-\frac{r}{\lambda}\right)}, \nonumber
	\end{align}			  
	where the convergence follows assuming $ s <n(r-1)+\frac{r\delta}{\sigma}$ and we choose $\rho$ to be such that 
	$
	\displaystyle{r+\frac{r\delta}{n}-\rho \left(r-1+\frac{r\delta}{n\sigma}\right)= \rho \left( 1-\frac{r}{2} \right)+ \frac{r}{p} \ \Leftrightarrow \ \rho := \frac{n \left( 1-\frac{1}{p}\right)+\delta}{\frac{n}{2}+\frac{\delta}{\sigma}}.}
	$ 
	By the choice of $\rho$ we have
	$$
	-n \left( 1-\frac{r}{2} \right)+\frac{nr}{1-\rho}\left( \frac{1}{p}-\frac{1}{2} \right) < n(r-1)+r\delta < n(r-1)+\frac{r\delta}{\sigma}.
	$$
	In particular, collecting the restrictions on $s$ we get
	$$
	n\left(\frac{r}{q}-1 \right)< s \leq -n \left( 1-\frac{r}{2} \right)+\frac{nr}{1-\rho}\left( \frac{1}{p}-\frac{1}{2} \right) \ \ \Rightarrow \ \ \frac{1}{q} < \frac{1}{2}+\dfrac{\beta \left( \frac{\delta}{\sigma}+\frac{n}{2} \right)}{n \left( \frac{\delta}{\sigma}-\delta+\beta \right)} := \frac{1}{q_{0}}. 
	$$
{We point out that when $\sigma=1$, only condition $s<n(r-1)+r\delta$ is imposed to verify $(M_1)$ and $(M_2)$.} Condition $(M_3)$, given formally by $T^{*}(x^{\alpha})=0$ for all $|\alpha| \leq  N_{q}$, is trivially valid, since $n/(n+\delta) < q_{0}< q \leq 1$ implies  $N_{q} \leq \lfloor \delta \rfloor $.
\end{proof}

\begin{remark}
 {\textnormal{The previous proof remains the same if one consider integral conditions incorporating derivatives of the kernel. For a complete discussion and the precise definition of such conditions see \cite[Section 4.2]{VasconcelosPicon}.}}
\end{remark}

%
%
%

\end{document}